\documentclass[11pt]{amsart}

\newcommand{\bX}{X}
\newcommand{\bY}{Y}

\usepackage{bm}  
\usepackage{hyperref}       
\usepackage{fullpage}       
\usepackage{url}            
\usepackage{booktabs}       
\usepackage{amsfonts}       
\usepackage{amssymb}        
\usepackage{nicefrac}       
\usepackage{microtype}      
\usepackage{tikz}           

\usepackage[textsize=tiny]{todonotes}

\usepackage[normalem]{ulem}
\newcommand{\stkout}[1]{\ifmmode\text{\sout{\ensuremath{#1}}}\else\sout{#1}\fi}

\usepackage{graphicx}
\graphicspath{{rawfigs/}}

\usepackage[ruled,vlined,linesnumbered]{algorithm2e}
\usepackage{amsmath}
\usepackage{amssymb}
\usepackage{amsthm}
\usepackage{bm}
\usepackage{graphicx}
\graphicspath{{figs/}}
\usepackage{mathtools}
\usepackage{subcaption}

\newcommand{\comment}[1]

\theoremstyle{remark}

\newtheorem{theorem}{Theorem}[section]
\newtheorem{lemma}[theorem]{Lemma}

\newtheorem{example}{Example}[section]

\begin{document}

\title{Diagonal nonlinear transformations preserve structure in covariance and precision matrices}

\author{Rebecca E. Morrison}
\address{Department of Computer Science,
        University of Colorado Boulder,
        Boulder, CO 80309, USA}
\email{rebeccam@colorado.edu}

\author{Ricardo Baptista}
\address{Center for Computational Engineering,
        Massachusetts Institute of Technology,
        Cambridge, MA 02139, USA}
\email{rsb@mit.edu}

\author{Estelle L. Basor}
\address{American Institute of Mathematics, San Jose, CA 95112, USA}
\email{ebasor@aimath.org}

\begin{abstract}
    For a multivariate normal distribution, the sparsity of the covariance and precision matrices
    encodes complete information about independence and conditional independence properties. For
    general distributions, the covariance and precision matrices reveal correlations and so-called
    partial correlations between variables, but these do not, in general, have any correspondence
    with respect to independence properties. In this paper, we prove that, for a certain class of
    non-Gaussian distributions, these correspondences still hold, exactly for the covariance and
    approximately for the precision. The distributions---sometimes referred to as
    ``nonparanormal''---are given by diagonal transformations of multivariate normal random variables. We provide
    several analytic and numerical examples illustrating these results.

\end{abstract}

\maketitle

\section{Introduction}\label{sec:int}
Among many appealing properties of multivariate normal distributions, their second moment matrix
and its inverse contain complete information about the independence and conditional independence
properties. Specifically, a zero in the $ij$th entry of the covariance matrix means that variables
$i$ and $j$ are marginally independent, while a zero in the $ij$th entry of the precision (inverse
covariance) means that the two are conditionally independent. For high-dimensional Gaussian data sets, it is
often of interest to estimate either a sparse covariance matrix, or sparse precision matrix, or, in
some cases, both~\cite{fan2016overview}.

In general, this correspondence---between the second moment matrix and the independence properties,
and between the inverse and conditional independence properties---does not hold for non-Gaussian
distributions. Tests to determine independence and conditional independence become more complex than
matrix estimation: the complexity of exhaustive pairwise testing techniques scales exponentially
with the number of variables \cite{lauritzen1996graphical}; other methods compute scores or combine
one-dimensional conditional distributions for the exponential family \cite{lin2016estimation,
yang2015graphical, suggala2017expxorcist}; another approach (by two of the current co-authors)
identifies conditional independence for arbitrary non-Gaussian distributions from the Hessian of the
log density, but is so far computationally limited to rather small graphs
\cite{baptista2021learning, morrison2017beyond}. Thus, it is of broad interest to analytically extract marginal and
conditional independence properties of a distribution \textit{a priori} to any estimation procedure.

In this paper, we show that the above correspondences between independence and sparsity of covariance and precision matrices are
approximately preserved for a broad class of distributions, namely those given by certain diagonal and
mean-preserving transformations of a multivariate normal (sometimes referred to as ``nonparanormal''
\cite{liu2009nonparanormal}). In particular, these distributions display the following behavior:
\begin{enumerate}
    \item Variables $i$ and $j$ are marginally independent if and only if the $ij$th entry of the covariance is
        zero, equivalent to the normal case.
    \item Variables $i$ and $j$ are conditionally independent if and only if the $ij$th entry of the
        precision is small, where ``small'' will be made precise later on.
\end{enumerate}
In other words, under some assumptions, a Gaussian approximation to a non-Gaussian distribution of
this form will exactly recover the marginal independence structure, and approximately recover the
conditional independence structure, which is often summarized as an undirected graphical model.  In
fact, numerical observations of the above phenomenon motivated the current work. In several
numerical examples in previous work~\cite{baptista2021learning}, algorithms built for Gaussian graph
learning recovered the graphs of non-Gaussian distributions, but without justification.  Trying to
explain why those algorithms still worked led to the results in this paper.

Covariance estimation for general non-Gaussian data sets is of course standard procedure, to (at least) identify
correlations between variables. Perhaps less common but not unusual is precision estimation for
general data sets, to identify a ``partial correlation graph''---a sort of first approximation of
the conditional independence properties \cite{koanantakool2018communication, rothman2008sparse,
friedman2008sparse, banerjee2008model}. 
This work provides a new mathematical foundation to connect the sparsity of computed partial
correlations to the conditional independence properties for a common class of non-Gaussian
distributions.

The rest of the paper is organized as follows. Section~\ref{sec:trans} proves exactly how the
entries of the covariance are transformed, with the result that the covariance structure is exactly
preserved by the diagonal transformation, and provides an explicit formula to calculate higher
moments of Gaussian random variables. Extra related computations are given in
Appendix~\ref{sec:app}. Section~\ref{sec:inv} computes the inverse covariance matrix after the
transformation, showing that the conditional independence structure is approximately preserved.
Numerical results for specific graphs are given in Section~\ref{sec:ex}, and we conclude with
Section~\ref{sec:con}.

\section{Moments after transformation}\label{sec:trans} Consider a multivariate normal random
variable $X \in \mathbb{R}^{d}$ with density $\rho = \mathcal{N}(0, \Sigma_\rho)$. Let $\Gamma_\rho
= \Sigma_\rho^{-1}$ be the inverse covariance or precision matrix of $X$. For $\rho$ satisfying some
conditional independence properties, the precision matrix $\Gamma_\rho$ will have zero entries.
Furthermore, the sparsity of $\Gamma_{\rho}$ defines the minimal I-map of $\rho$, i.e., the
minimal undirected graphical model satisfying the conditional independence properties of $\rho$.

Now apply a (nonlinear) univariate transformation $f\colon \mathbb{R} \rightarrow \mathbb{R}$ to
each element of $X$, i.e., $Y_i = f(X_i)$. We refer to the overall mapping as a diagonal
transformation of $X$ and we denote this by $f(X)$.
Let us then say that $Y = f(X) \sim \pi$ where $\pi$ is the push-forward density of $\rho$ through
the diagonal transformation. For a general nonlinear $f$, $Y$ will have a non-Gaussian distribution.
Nevertheless, its mean is given by $\mu_{\pi} = \mathbb{E}_{\pi}[\bY] = \mathbb{E}_{\rho}[f(\bX)]$
and its covariance is given by $\Sigma_{\pi} = \mathbb{E}_{\pi}[(\bY - \mu_\pi)(\bY - \mu_\pi)^{T}]
= \mathbb{E}_{\rho}[f(\bX)f(\bX)^T] - \mu_\pi\mu_\pi^{T}$.
For simplicity and without loss of generality, we will assume that $f$ is
mean-preserving, so that $\mu_{\pi} = 0$ and $\Sigma_\pi =
\mathbb{E}_\rho[f(X)f(X)^T]$. We let $\Gamma_\pi = \Sigma_\pi^{-1}$ denote the precision matrix of
$\pi$. Also note that a moment-matching Gaussian approximation to $\pi$ will have mean $\mu_\pi$ and
covariance $\Sigma_\pi$.

Our ultimate goal is to derive conditions under which a Gaussian approximation to $\bY$
will approximately preserve the conditional independence properties (i.e., the corresponding entries
of the inverse covariance matrix of $\bY$ will be small). To do so, in this section we first
characterize the first and second moments of $Y$.

\subsection{Univariate moments after transformation}
If $f$ is a smooth function of $X_{i}$, we can expand it using a Taylor series expansion around
$\mathbb{E}[X_{i}] = 0$ as
\begin{equation} \label{eq:taylor_series}
    f(x) = \sum_{k=0}^{\infty} \frac{f^{(k)}(0)}{k!} x^{k}.
\end{equation}
Using this expansion we compute the first two moments of $Y_{i} = f(X_{i})$ by taking advantage
of the linearity of the expectation operator. They are
\begin{equation} \label{eq:mean_expansion}
    \mathbb{E}_{\rho}[f(X_{i})] = \mathbb{E} \left[ \sum_{k=0}^{\infty} \frac{f^{(k)}(0)}{k!}
    X_{i}^{k} \right] = \sum_{k=0}^{\infty} \frac{f^{(k)}(0)}{k!} \mathbb{E}_{\rho}[X_{i}^{k}],
\end{equation}
\begin{equation} \label{eq:var_expansion}
    \mathbb{E}_{\rho}[f(X_{i})^2] = \mathbb{E} \left[\left( \sum_{k=0}^{\infty}
    \frac{f^{(k)}(0)}{k!} X_{i}^{k} \right) \left(\sum_{l=0}^{\infty} \frac{f^{(l)}(0)}{l!}
    X_{i}^{l} \right)\right] = \sum_{k=0}^{\infty} \sum_{l=0}^{\infty}
    \frac{f^{(k)}(0)f^{(l)}(0)}{k!l!} \mathbb{E}_{\rho}[X_{i}^{k+l}].
\end{equation}
For a mean-zero Gaussian random variable $X_{i}$ with variance $(\Sigma_{\rho})_{ii} \coloneqq
\sigma_{ii}^2$, its moments are given by
\begin{equation}\label{moment}
    \mathbb{E}_{\rho}[X_{i}^{k}] = \left\{ \begin{array}{ll} 0 & \text{if $k$ is odd} \\
    \sigma_{ii}^{k}(k-1)!! & \text{if $k$ is even} \end{array} \right.,
\end{equation}
where $k!!$ denotes the double factorial of $k$. Next we ignore any terms
in~\eqref{eq:mean_expansion} and~\eqref{eq:var_expansion} that involve odd terms in the exponent $k$ and $k+l$.
Furthermore, we can write the first and second moments only in terms of the variance of $X_{i}$. The
first moment is given by
\begin{equation} \label{eq:univariate_first_moment}
    \mathbb{E}_{\rho}[f(X_{i})] = \sum_{\substack{k=0 \\ \text{even } k}}^{\infty}
    \frac{f^{(k)}(0)}{k!} \mathbb{E}_{\rho}[X_{i}^{k}] = \sum_{\substack{k=0 \\ \text{even } k}}^{\infty} a_{k} \sigma_{ii}^{k},
\end{equation}
where $a_{k} \coloneqq f^{(k)}(0) \frac{(k-1)!!}{k!}$ is a constant depending on the higher-order
derivatives of $f$ and the index $k$. Similarly, after a re-parameterization of the indices over the
sum $n=k+l$, the second moment is given by
\begin{align}
    \mathbb{E}_{\rho}[f(X_{i})^2] &= \sum_{n=0}^{\infty} \sum_{p=0}^{n}
    \frac{f^{(p)}(0)f^{(n-p)}(0)}{p!(n-p)!} \mathbb{E}_{\rho}[X_{i}^{n}] \nonumber \\
    &= \sum_{\substack{n=0 \\ \text{even } n}}^{\infty} \mathbb{E}_{\rho}[X_{i}^{n}] \left(\sum_{p=0}^{n}
        \frac{f^{(p)}(0)f^{(n-p)}(0)}{p!(n-p)!} \right)
    = \sum_{\substack{n=0 \\ \text{even } n}}^{\infty} b_{n} \sigma_{ii}^{n}, \label{eq:univariate_second_moment}
\end{align}
where $b_{n} \coloneqq (n-1)!!\left(\sum_{p=0}^{n} \frac{f^{(p)}(0)f^{(n-p)}(0)}{p!(n-p)!}\right) =
(n-1)!!\frac{g^{(n)}(0)}{n!}$ with $g(x) = f^{2}(x)$. 

Of course, the above computations assume convergence of each series. In the next section we will
give criteria for these series to converge.

\subsection{Multivariate moments after transformation}
With a similar argument as above, we can also derive the second moment matrix of the transformation.
Using the Taylor series expansions in~\eqref{eq:taylor_series}, the second moment matrix is given by
\begin{align}
    \mathbb{E}_{\rho}[f(\bX)f(\bX)^{T}] &= \mathbb{E}_\rho\left[ \sum_{k=0}^{\infty}
    \sum_{l=0}^{\infty} \frac{f^{(k)}(0)f^{(l)}(0)}{k!l!} X^{k}(X^{l})^{T} \right] \nonumber \\
    &=    \sum_{n=0}^{\infty} \sum_{p=0}^{n} \frac{f^{(p)}(0)f^{(n-p)}(0)}{p!(n-p)!} \mathbb{E}_{\rho}
    [X^{p}(X^{n-p})^{T}] \label{eq:multivariate_moment}
\end{align}
where in the second line we re-parameterize and switch the order of summations and expectation.

Using Isserlis' theorem (or Wick's probability theorem)~\cite{isserlis1918formula}, we compute
the moments of the product of Gaussian random variables as:
\begin{align} \label{eq:Wicks_thm}
    \mathbb{E}_{\rho} &[X_{i}^{p}X_{j}^{n-p}] \nonumber \\
    &=\left\{ \begin{array}{ll} 0 & \text{$n$ odd}\\ \vspace{1em}
    \sum_{\substack{k = p\\ \text{ by }-2}}^0 (p-k-1)!! {p \choose k}{n-p \choose k} k! (n-p-k-1)!!\, \sigma_{ii}^{(p-k)/2} \sigma_{ij}^k
        \sigma_{jj}^{(n-p-k)/2} & \text{$n$ even, $p \leq n/2$}\\ \vspace{1em}
    \sum_{\substack{k = n-p\\ \text{ by }-2}}^0 (p-k-1)!! {p \choose k}{n-p \choose k}
    k! (n-p-k-1)!!\, \sigma_{ii}^{(p-k)/2} \sigma_{ij}^k
        \sigma_{jj}^{(n-p-k)/2} & \text{$n$ even, $p > n/2$}
    \end{array} \right. .
\end{align}


\begin{example}
For the transformation $f(x) = x^3$, the third order truncation of the expansion
    in~\eqref{eq:multivariate_moment} is exact and the derivatives evaluated at zero are all zero
    with the exception of $f'''(0) = 6$. Using this result together with $\mu_\rho = \mathbb{E}_\rho[f(X)] = 0$, we have that $$(\Sigma_{\pi})_{ij} =
    \frac{f'''(0)^2}{36} \mathbb{E}_\rho[\bX_{i}^3\bX_{j}^3] = \mathbb{E}_\rho[\bX_{i}^3\bX_{j}^3] =
    9\sigma_{ii}\sigma_{jj}\sigma_{ij} + 6\sigma_{ij}^3.$$ 
\end{example}


Just as what was done above for $f(x) = x^{3},$ one of our goals is to compute the general explicit
form of $(\Sigma_{\pi})_{ij}$. Before we go any further, however, we should prove that the above
series actually makes sense, that is, that it converges.  We do this now.
\begin{lemma}
Suppose that the derivatives of the function $f$ are all bounded at zero. Then the series
$$ \sum_{n\geq2} \sum_{p=0}^{n} \frac{f^{(p)}(0)f^{(n-p)}(0)}{p!(n-p)!}
    \mathbb{E}_{\rho}[\bX^{p}(\bX^{n-p})^{T}], $$
    converges.
\end{lemma}
\proof
We consider 
\begin{equation} \label{eq:sum_testing_convergence}
\sum_{n \geq 2} \sum_{p=0}^{n} \frac{f^{(p)}(0)f^{(n-p)}(0)}{p!(n-p)!}\mathbb{E}_{\rho}[X_{i}^{p}X_{j}^{n-p}]
\end{equation}
for fixed $i$ and $j.$ 
To begin, we use the Cauchy-Schwartz inequality with
\begin{equation*} 
\mathbb{E}_{\rho}[X_{i}^{p}X_{j}^{n-p}] = \int X_{i}^{p}X_{j}^{n-p} \rho(X) \,\textrm{d}X.
\end{equation*}
The square of this expectation is bounded by the product of integrals
\begin{equation*}
\left(\int X_{i}^{2p} \rho(X) \,\textrm{d}X \right) \left(\int X_{j}^{2n-2p} \rho(X) \,\textrm{d}X \right).
\end{equation*}
From~\eqref{moment} we know that this last product is then
\[ \sigma_{ii}^{2p} (2p-1)!! \,\,\sigma_{jj}^{2(n-p)}(2(n-p)-1)!! .\]

Let $M$ be a bound on $\sigma_{ii}$ for all $i$ and recall we are assuming that the derivatives are
bounded. Then the sum in~\eqref{eq:sum_testing_convergence} is bounded by
\begin{equation} \label{eq:series_bounded_sum}
\sum_{n \geq 2} \sum_{p=0}^{n} \frac{((2p-1)!! (2(n-p)-1)!!)^{1/2}}{p!(n-p)!}M^{n} .
\end{equation}
Note that $$(2p-1)!! = \frac{(2p-1)!}{(p-1)! 2^{p-1}} = \frac{(2p)!}{p! 2^{p}}  = \Gamma( p + 1/2)
2^{p} \pi^{-1/2},$$ where the last equality follows from the duplication formula for the Gamma
function $\Gamma$. Thus, the sum in~\eqref{eq:series_bounded_sum} is given by
\begin{equation*}
\sum_{n \geq 2}  2^{n/2} M^{n} \pi^{-1/2}\sum_{p=0}^{n} \frac{(\Gamma( p +1/2) \Gamma( n - p + 1/2))^{1/2}}{p!(n-p)!}.
\end{equation*}
Let us consider the inner sum over $p.$ Using the symmetry in $p$ and $n-p,$ this is at most
\begin{equation} \label{eq:sum_gamma_functions}
    2\,\sum_{p=0}^{\left \lfloor{\frac{n}{2}}\right\rfloor + 1} \frac{(\Gamma( p +1/2) \Gamma( n - p + 1/2))^{1/2}}{p!(n-p)!}.
\end{equation}
The ratio $$\frac{\Gamma(k + 1/2)^{1/2}}{\Gamma(k + 1)}$$ decreases as $k$ increases and hence the sum in~\eqref{eq:sum_gamma_functions}
is bounded by
\begin{equation*}
    2\,\frac{\Gamma(n - \left\lfloor{n/2}\right\rfloor - 1/2)^{1/2}}{\Gamma(n
    -\left\lfloor{n/2}\right\rfloor)}\sum_{p=0}^{\infty} \frac{\Gamma( p +1/2)^{1/2}}{p!}.
\end{equation*}
The sum over $p$ converges and hence it is bounded. The convergence is easily seen with the ratio
test. Thus we are left with a constant multiplied by
\begin{equation*}
    \sum_{n \geq 2}  2^{n/2} M^{n} \,\frac{\Gamma(n - \left \lfloor{n/2}\right\rfloor  -
    1/2)^{1/2}}{\Gamma(n - \left \lfloor{n/2}\right \rfloor )},
\end{equation*}
which also converges by the ratio test after grouping each even indexed term with the following odd
indexed term.
\endproof
To end this section we note that the derivatives of $f$ evaluated at zero need not be bounded to establish convergence. If
the derivatives 
grow as a power $n^{\alpha}$ or even geometrically as
$k^{n}$ the same computation would work.

\subsection{Transformation of matrix elements}
This focus of this section is to describe how individual matrix elements of the covariance matrix
get transformed when $f$ is applied to the random variable $X$. We assume that $f(0) = 0$ and thus $g(0) = 0$. For the
diagonal entries of the covariance, we already know the answer from
(\ref{eq:univariate_second_moment}). Let $(\Sigma_{\pi})_{ii} = \tau_{ii}.$ 
Then 
$$\tau_{ii} = \sum_{\substack{n=2 \\ \text{even } n}}^{\infty}(n-1)!!\frac{g^{(n)}(0)}{n!}  \sigma_{ii}^{n},$$ 
where $g(x) = f^{2}(x)$. This means that if $\sigma_{ii} = 1$ (or is close to one) then the diagonal elements of
the transformed covariance are all equal to a constant factor that only depends on the derivatives of $f$ at zero. Here is an example. 

\begin{example}
Suppose $f(x) = \sin(x)$, and that $\sigma_{ii} = 1.$ Then $f^2(x) = \frac{1 - \cos 2 x}{2}$ and
we have $$\tau_{ii} = \sum_{\substack{n=2 \\ \text{even } n}}^{\infty}(n-1)!!\frac{g^{(n)}(0)}{n!} =
\sum_{k=1}^{\infty}\frac{(-1)^{k+1}2^{k-1}(2k)!}{k!\,(2k)!} = \frac{ 1- e^{-2}}{2} \text{ for all $i$.}$$ 
\end{example}

We now examine what happens to the other elements of the covariance after the transformation. We
will consider for now the case of an odd function $f$. 

\begin{theorem}\label{thm:main}
Suppose $f$ is given by 
\[ f(x)  = \sum_{u \geq 0} \frac{f^{(2u + 1)}(0)}{ (2u + 1)!} x^{2u + 1}.\] Define a new function
\[ F_{k}(x)  =  \sum_{u \geq 0} \frac{f^{(2u + k)}(0)}{ u!} x^{u},\]
and $G_{kij}(x) = F_{k}(\sigma_{ii}x)F_{k}(\sigma_{jj}x).$
Then $\sigma_{ij}$ is transformed to
\begin{equation} \label{eq:transformed_moment}
\tau_{ij} = \sum_{\text{odd} \,\,k} G_{kij}(1/2) \frac{\sigma_{ij}^{k}}{k!}.
\end{equation}
\end{theorem}
\proof
First note that for an odd function $f$, $f^{(l)}(0) = 0$ if $l$ is even. After inserting formula~\eqref{eq:Wicks_thm} in~\eqref{eq:multivariate_moment}, we combine all the terms that correspond to $\sigma_{ij}^{k}$. By
simplifying the double factorials, one has that the coefficient of $\sigma_{ij}^{k}$ is given by
\begin{equation} \label{eq:coefficient_sigmaijk}
\sum_{\substack{n \,\,\text{even} \\ n/2 \geq k}} \sum_{\substack{p \,\,\text{odd}\\ p =k}}^{n - k} \frac{f^{(p)}(0)\,\, f^{(n-p)}(0) \,\,\sigma_{ii}^{\frac{p-k}{2}}\,\, \sigma_{jj}^{\frac{n-p-k}{2}}}{k!\,\, 2^{n/2 - k} \,\,(\frac{p-k}{2})! \,\,(\frac{n -p-k}{2})!}.
\end{equation} 
Let $n = 2m$, $p = 2r+1$, and $k = 2s+1$. Then the sum in~\eqref{eq:coefficient_sigmaijk} becomes 
\begin{equation*}
\frac{1}{k! \,\, 2^{-k}} \sum_{m \geq k} \frac{1}{2^{m}} \sum_{ r = s}^{m - s -1} \frac{f^{(2 r +
    1)}(0)\,\, f^{(2 m - 2r - 1)}\sigma_{ii}^{r-s}\,\, \sigma_{jj}^{m - r -s -1}}
{\,\, (r - s)! \,\,( m - r - s -1)!}.
\end{equation*}
Again change variables. Let $r- s = t,$ and we have
\begin{equation*}
\frac{1}{k! \,\, 2^{ -k}} \sum_{m \geq k} \frac{1}{2^{m}} \sum_{ t = 0}^{m - 2s -1} \frac{f^{(2 (s +
    t) + 1)}(0)\,\, f^{(2 m - 2(s + t) - 1)}(0) \sigma_{ii}^{t} \,\,\sigma_{jj}^{m - t - 2s
    -1}}{\,\, t! \,\,( m -t - 2 s -1)!}.
\end{equation*}
 Recall our functions
\[ F_{k}(x)  =  \sum_{u = 0} \frac{f^{(2u + k)}(0)}{ u!} x^{u},\]
and $G_{kij}(x) = F_{k}(\sigma_{ii}x)F_{k}(\sigma_{jj}x).$
Then the coefficient of $\sigma_{ij}^k$ becomes
\[ \frac{1}{k! \,\, 2^{ -k}} \sum_{m \geq k} \frac{1}{2^{m}}\frac{G_{kij}^{(m - k)}(0)}{ ( m -k)!}. \]
Finally, by using a Taylor series expansion of $G_{kij}$ around $0$ we can write this coefficient as
\[ \frac{1}{k!} \sum_{m = 0} \frac{1}{2^{m}}\frac{G_{kij}^{(m)}(0)}{  m!} = \frac{1}{k!} G_{kij}(1/2),\] and the result follows.
\endproof

While the answer may look complex, it is easy to compute in many cases. The computation is also
often simplified by noting that, for $k=2s + 1$, $F_{k}$ is the $s$th derivative of $F_{1}.$ Here
are some examples: 
\begin{enumerate}
\item 
Let $f(x) = \sin(x).$ Then $f^{(2u +1)}(0) = (-1)^{u},$ and $F_{k}(x) =  (-1)^{\frac{k-1}{2}}e^{-x}$. Thus $G_{kij}(x) = e^{- (\sigma_{ii} + \sigma_{jj})x}$. Summing over the odd indices $k$ we have that $\sigma_{ij}$ is transformed to 
$$\tau_{ij} = e^{-\frac{\sigma_{ii}+\sigma_{jj}}{2}}\,\sinh \sigma_{ij}.$$ 
\item 
Let us also verify the computation done earlier for the function $f(x) = x^{3}.$
For this case, $F_{1}(x) = 6x$, $F_{3}(x) = 6,$ and $F_{k} = 0$ if $k > 3$.
Thus we have $G_{1ij}(x) = 36(\sigma_{ii}\sigma_{jj}x^{2})$, and $G_{3ij}(x) = 36$. This yields 
a final answer of 
$$\tau_{ij} = 9 \sigma_{ii}\sigma_{ij}\sigma_{jj} + 6\sigma_{ij}^{3}.$$
\item 
Let $f(x) = \sinh(x).$ Then $f^{(2u +1)}(0) = 1,$ and $F_{k}(x) =  e^{x}.$ Thus $G_{kij}(x) = e^{(\sigma_{ii}+\sigma_{jj})x} ,$ and 
we have that $\sigma_{ij}$ is transformed to 
$$\tau_{ij} = e^{\frac{\sigma_{ii}+\sigma_{jj}}{2}}\,\sinh \sigma_{ij}.$$  
\item
Let $f(x) = \frac{x^{2l+1}}{(2l+1)!}$ Then $F_{1}(x) = \frac{x^{l}}{l!}$ and $F_{2s+1}(x) = \frac{x^{l - s}}{(l-s)!}.$ Thus we have
        $$\tau_{ij} = \sum_{s = 0}^{l} \frac{(\sigma_{ii}\sigma_{jj})^{l-s}}{4^{(l-s)}((l-s)!)^2(2s+1)!}\sigma_{ij}^{2s+1}.$$
\end{enumerate}

In Appendix~\ref{sec:app}, we also provide related and somewhat simpler computations that may be
helpful in some circumstances. In particular, we compute the coefficient of the linear term in
expansion~\eqref{eq:multivariate_moment}, i.e., the coefficient corresponding to $\sigma_{ij}^1$.

To conclude this section, we present an important (and well-known) consequence of
Theorem~\ref{thm:main} about the marginal independence properties of random variables after diagonal
transformations. For a pair of mean-zero Gaussian variables $(X_i,X_j)$ that are uncorrelated (and
hence also marginally independent), we have $\sigma_{ij} = 0$. Thus,
by~\eqref{eq:univariate_first_moment}
and~\eqref{eq:transformed_moment} we have that $\tau_{ij} = (\Sigma_{\pi})_{ij} = 0$ and the variable pair
$(Y_i,Y_j)$ is also uncorrelated\footnote{In fact, diagonal transformations also preserve the
marginal independence properties of the random variables.}. 
Therefore, diagonal transformations exactly preserve the sparsity of the covariance matrix for $X$,
i.e., the zero elements in $\Sigma_\rho$, in the covariance matrix of $Y$.

\section{Properties of the inverse covariance}\label{sec:inv} Our goal is to not only say something
about the covariance matrix for the transformed variables, but also about what happens to the
entries in the precision matrix. We may be faced with the following situation. We have a precision
matrix from a multivariate normal variable with some zeros to start. After applying a diagonal transformation to those
variables, we would like to know what (approximate) sparsity can be recaptured in the inverse
covariance matrix of the non-Gaussian variables. We probably cannot hope to do this in general, but
we can say something specific about some particular cases that often do occur in applications.
First, we begin with a technical lemma about matrix inverses.

\begin{lemma} 
Let $A = I + B$ where the operator norm of $B$, $\|B\|$, is at most $\delta < 1$. Then $A^{-1} = I
    - B + E$ where the norm of $E$ is at most $\frac{\delta^{2}}{1 - \delta}$. 
\end{lemma}
\begin{proof}
From the Neumann series expansion for the inverse of $A$, we have
\begin{equation*}
    A^{-1} = \sum_{k = 0}^{\infty}(-B)^{k} = I  - B + \sum_{k = 2}^{\infty}(-B)^{k}.
\end{equation*} 
For $\|B\| < \delta$, the norm of the last term is at most $\frac{\delta^{2}}{1 - \delta}$ for a converging geometric series. 
 \end{proof}
We note that entry $(i,j)$ of the matrix $A^{-1}$ must be of the form $\delta_{ij} - B_{ij} +
E_{ij}$ where $|E_{ij}| \leq \|E\|$ and thus 
$|E_{ij}| \leq \frac{\delta^2}{1 - \delta}$. Lastly,
we note that the inverse matrix can also be written as $A^{-1} = I - B^{'}$ where the operator norm of
$B^{'} \coloneqq B - E$ is at most $\epsilon \coloneqq \frac{\delta}{1 - \delta}.$
 
Now suppose that we have a precision matrix of the form $I + B$ where $B$ has diagonal elements
equal to $0$ and $\|B\| < \delta.$ By the lemma above, whenever an off-diagonal entry of the
precision matrix is zero, the corresponding entry of the covariance will be of order (at most)
$\epsilon^{2}$. If the entry in the precision is not zero, then the entry is of the from
$\delta_{ij} - B_{ij}$ plus something again of order $\epsilon^{2}$.  
The purpose of the next three lemmas is to get good estimates for the entries of the covariance
matrix after the transformation, when the function $f$ is odd. First,
Lemma~\ref{lem:transf_cov_oddf} estimates the off-diagonal $(i,j)$ entries, but the
scaling relative to $\sigma_{ij}$ still depends on $i$ and $j$. Second,
Lemma~\ref{lem:cov-diag} estimates the diagonal $(i,i)$ terms, where the scaling factor for the leading term is
independent of $i$. Third, Lemma~\ref{lem:cov-offd} modifies the estimates of off-diagonal terms, so
that the scaling factor of the leading terms no longer depends on $i$ and
$j$.

\begin{lemma} \label{lem:transf_cov_oddf}
Suppose that $f$ is an odd function with derivatives at zero bounded by $N$, and the precision matrix $\Gamma_\rho$ is of
    the form $I + B$ where $B$ has operator norm at most $\delta < 1$. Then for $i \neq j$
\begin{equation}
    \tau_{ij} = G_{1ij}(1/2)\sigma_{ij} + O(\epsilon^{3}),
\end{equation}
    where $G_{1ij}(1/2)$ is a constant that depends only on $f$ that is given in Theorem \ref{thm:main}, and
    $\epsilon = \frac{\delta}{1-\delta}$.
\end{lemma}
\proof
From Theorem~\ref{thm:main} we have that the transformation of $\sigma_{ij}$ is given by 
 \[ \tau_{ij} = \sum_{\text{odd} \,\,k} G_{kij}(1/2) \frac{\sigma_{ij}^{k}}{k!}.\] If $N$ is a bound
 on the derivatives of $f$ at zero, then 
 it follows that \[|G_{kij}(1/2)| \leq N^2 e^{(\sigma_{ii} + \sigma_{jj})/2} \leq N^2e^{ 1 +
 \epsilon}, \]
where the last inequality follows from the diagonal entries of $\Sigma_\rho = I - B'$ being bounded by $1 + \epsilon$. 
Next, the difference $|\tau_{ij}  - G_{1ij}(1/2)\sigma_{ij}|$ for any $i \neq j$ is bounded as:
 \begin{align*} 
     |\tau_{ij}  - G_{1ij}(1/2)\sigma_{ij}| &\leq \sum_{\substack{\text{odd}\,\,k\\ k \geq 3}}
     \left|G_{kij}(1/2) \frac{\sigma_{ij}^k}{k!}\right| \\
     &\leq N^2 e^{1+ \epsilon} \epsilon^3 \sum_{\substack{\text{odd}\,\,k\\ k \geq 3}} \frac{|\sigma_{ij}|^{k-3}}{k!}\\
     &\leq N^2 e^{1+ \epsilon} \frac{\epsilon^3}{3!} \sum_{l\geq0} \frac{|\sigma_{ij}|^{2l}}{(2l)!}
     \leq N^2 e^{1+ \epsilon} \frac{\epsilon^3}{6} \cosh \epsilon,
 \end{align*}
using that off-diagonal entries of $\Sigma_\rho$ are bounded by $\epsilon$. Given that $\epsilon <
1,$ we have a bound for the difference of at most $2N^2\epsilon^3$, although for a given function
and smaller $\epsilon$ this bound can be improved.
\endproof

\begin{lemma}\label{lem:cov-diag}
Suppose that $f$ is an odd function with bounded derivatives at zero, and the precision matrix is of
    the form $I + B$ where $B$ has norm at most $\delta <1,$  $\epsilon$ is given as above, and
    $B_{ii} = 0.$ Let 
\[ \kappa = \sum_{\text{odd }\,k} \frac{F_{k}^{2}(1/2)}{k!}.\] Then 
$$\tau_{ii} = \kappa \sigma_{ii} + O(\epsilon^{2}).$$
\end{lemma} 
\proof
The proof of this is similar to the previous lemma. Notice that the transformation of $\sigma_{ii}$
is given by
$$ \tau_{ii} = \sum_{\text{odd } k} G_{kii}(1/2) \frac{\sigma_{ii}^{k}}{k!}.$$ We think of this as a function of
the variable $\sigma_{ii}$. When $\sigma_{ii}$ is one, then this evaluates to $\kappa$. 
Recall that $\sigma_{ii} = 1 + O(\epsilon^{2})$ and thus using Taylor's theorem (since all
derivatives of this function are bounded in a neighborhood of one), the result follows.
\endproof

\begin{lemma}\label{lem:cov-offd}
Suppose that $f$ is an odd function with bounded derivatives at zero, the precision matrix is of the form $I
    + B$ where $B$ has norm at most $\delta <1$,  $\epsilon$ is given as above, and $B_{ii} = 0.$
    Then for $i \neq j$
\begin{equation*}
\tau_{ij} = \lambda\,\sigma_{ij} + O(\epsilon^{3}).
\end{equation*}
where $\lambda = F_{1}^{2}(1/2).$
\end{lemma}
\begin{proof}
To see this, notice that by Taylor's remainder theorem,
\[ \left|F_1(x/2) - F_1(1/2)\right| \leq \frac{\alpha}{2} |x-1| \] for some
constant $\alpha$ in a small neighborhood around $x=1$. Now let $x = \sigma_{ii} = 1 +
O(\epsilon^2)$, which relies on $B_{ii} = 0$ (otherwise we would have $O(\epsilon)$). Thus
\[ \left|F_1(x/2) - F_1(1/2)\right| \leq \frac{\alpha}{2} O(\epsilon^2) .\]

We now replace $G_{1ij}(1/2)$ in Lemma~\ref{lem:transf_cov_oddf} with the approximation
\begin{align*}G_{1ij}(1/2) &= F_1(\sigma_{ii} /2) F_1(\sigma_{jj}/2)\\
    &=\left(F_1(1/2) + O(\epsilon^2) \right)\left(F_1(1/2) + O(\epsilon^2) \right)\\
    &= F_1^2(1/2) + O(\epsilon^2).
\end{align*}
The difference above is $O(\epsilon^{2})$, and since $\sigma_{ij}$ is at most $\epsilon$, the result
follows.
\end{proof}

To summarize, in the case of odd functions with the precision matrix $\Sigma_\rho^{-1}$ given above,
we have that the transformed covariance matrix $\Sigma_\pi$ has the form
\begin{equation} \label{eq:transformed_covmatrix_oddf}
\tau = \kappa I - \lambda B + E^{'}
\end{equation}
 where $\kappa$ and $\lambda$ are given in
 Lemmas~\ref{lem:cov-diag} and~\ref{lem:cov-offd}, respectively, $B$ is as before, and $E^{'}$ is
 the error. Let us now estimate the norm of the error term $E^{'}$. 
\begin{theorem} \label{thm:error_term}
Let the precision matrix be a $d\times d$ matrix, and suppose $d \epsilon$ is bounded. Then the
    operator norm of $E^{'}$ is at most $O(\epsilon^{2})$. 
\end{theorem}
\begin{proof}
    Let us first split $E^{'}$ into a diagonal matrix $E_1$, and a matrix with only off-diagonal
    non-zero elements $E_2$, i.e., $E^{'} = E_1 + E_2$.
 The entries of $E_1$ are $O(\epsilon^{2})$ and hence this matrix has operator norm at most
    $O(\epsilon^{2})$. The entries of $E_2$ are
    $O(\epsilon^{3})$; thus this matrix has Hilbert-Schmidt
    norm at most $O(d\epsilon^{3})$.  Finally,
    \begin{align*} \| E^{'} \| \leq \|E_1\| + \|E_2\| \leq \|E_1 \| + \|E_2\|_{\text{HS}} &= O(\epsilon^2) + O(d
    \epsilon^3)\\
        &=O(\epsilon^2),
    \end{align*}
    where, in the final line we have used that $d\epsilon$ is bounded.
 \end{proof}
 
Note that these results can be made much more precise for specific functions $f$ and for specific
matrices of fixed size when it is possible to keep track of the constants in the $O$ estimates.  On
the other hand, the result in Theorem~\ref{thm:error_term} allows for arbitrarily large matrices.
The bounded  condition on $d\epsilon$ (as opposed to just fixing the dimension $d$) is also
intuitive in the sense that edge weights 
often decrease as graphs grow larger in dimension (think, for example, of a star graph). Future work
will explore the scaling dependence of the error term in Theorem~\ref{thm:error_term} on other
graph properties such as the maximum node degree.
 
Our final step is to compute the inverse of the covariance matrix
in~\eqref{eq:transformed_covmatrix_oddf}, which is given by 
\begin{align} \tau^{-1} &=\kappa^{-1} ( I - \frac{\lambda}{\kappa}(B-  E^{'}) )^{-1} \nonumber \\
&=\kappa^{-1} ( I + \frac{\lambda}{\kappa}B + E^{''}) = \frac{1}{\kappa}I + \frac{\lambda}{\kappa^{2}}B + \frac{1}{\kappa}E^{''} \nonumber\end{align}
for some error term $E''$. 
Thus the off-diagonal terms of the transformed precision have much the same behavior of the original precision matrix. If
$B_{ij} = 0$, the off-diagonal entries are of order $\epsilon^{2}$. If they are not zero, then the first order
term scales with $B_{ij}$.

This will be illustrated by examples in the next section.

\section{Applications to specific graphs}\label{sec:ex}
\subsection{Chain graph}
We begin with an example of a circulant matrix that corresponds to the starting precision of a chain
graph. Consider $$\Gamma_\rho = \left(
\begin{array}{cccccccc}
 1 & \frac{1}{22} & 0 & 0 & 0 & 0 & 0 & \frac{1}{22} \\
 \frac{1}{22} & 1 & \frac{1}{22} & 0 & 0 & 0 & 0 & 0 \\
 0 & \frac{1}{22} & 1 & \frac{1}{22} & 0 & 0 & 0 & 0 \\
 0 & 0 & \frac{1}{22} & 1 & \frac{1}{22} & 0 & 0 & 0 \\
 0 & 0 & 0 & \frac{1}{22} & 1 & \frac{1}{22} & 0 & 0 \\
 0 & 0 & 0 & 0 & \frac{1}{22} & 1 & \frac{1}{22} & 0 \\
 0 & 0 & 0 & 0 & 0 & \frac{1}{22} & 1 & \frac{1}{22} \\
 \frac{1}{22} & 0 & 0 & 0 & 0 & 0 & \frac{1}{22} & 1 \\
\end{array}
\right).$$
For this $\Gamma_\rho$, $\delta = 1/11.$ Its inverse is given by (rounded to $10^{-4}$ places)
$$\Gamma_\rho^{-1} = \left(
\begin{array}{cccccccc}
 1.0042 & -0.0457 & 0.0021 & -0.0001 & 0. & -0.0001 & 0.0021 & -0.0457 \\
 -0.0457 & 1.0042 & -0.0457 & 0.0021 & -0.0001 & 0. & -0.0001 & 0.0021 \\
 0.0021 & -0.0457 & 1.0042 & -0.0457 & 0.0021 & -0.0001 & 0. & -0.0001 \\
 -0.0001 & 0.0021 & -0.0457 & 1.0042 & -0.0457 & 0.0021 & -0.0001 & 0. \\
 0. & -0.0001 & 0.0021 & -0.0457 & 1.0042 & -0.0457 & 0.0021 & -0.0001 \\
 -0.0001 & 0. & -0.0001 & 0.0021 & -0.0457 & 1.0042 & -0.0457 & 0.0021 \\
 0.0021 & -0.0001 & 0. & -0.0001 & 0.0021 & -0.0457 & 1.0042 & -0.0457 \\
 -0.0457 & 0.0021 & -0.0001 & 0. & -0.0001 & 0.0021 & -0.0457 & 1.0042 \\
\end{array}
\right)$$
and now $\epsilon = \delta/(1 - \delta) = 1/10$.
We now demonstrate the effect of applying the diagonal transformation $f(x) = \sin(x)$ to a multivariate normal vector $X$ with the covariance above. Using Theorem~\ref{thm:main}, the transformed covariance is given by
$$
\Sigma_\pi = \left(
\begin{array}{cccccccc}
 0.4329 & -0.0168 & 0.0008 & 0. & 0. & 0. & 0.0008 & -0.0168 \\
 -0.0168 & 0.4329 & -0.0168 & 0.0008 & 0. & 0. & 0. & 0.0008 \\
 0.0008 & -0.0168 & 0.4329 & -0.0168 & 0.0008 & 0. & 0. & 0. \\
 0. & 0.0008 & -0.0168 & 0.4329 & -0.0168 & 0.0008 & 0. & 0. \\
 0. & 0. & 0.0008 & -0.0168 & 0.4329 & -0.0168 & 0.0008 & 0. \\
 0. & 0. & 0. & 0.0008 & -0.0168 & 0.4329 & -0.0168 & 0.0008 \\
 0.0008 & 0. & 0. & 0. & 0.0008 & -0.0168 & 0.4329 & -0.0168 \\
 -0.0168 & 0.0008 & 0. & 0. & 0. & 0.0008 & -0.0168 & 0.4329 \\
\end{array}
\right).
$$
We note that as predicted, this matrix is circulant
and preserves the sparsity in the covariance matrix of $X$. 

To verify the computation of $\Sigma_\pi$, we also estimate the covariance using samples. To do so,
we generated 100,000 samples from the distribution with the above covariance and then applied the
function $y = \sin x$ to the data.  The resulting empirical covariance (which does not preserve the
circulant property) was
$$
\left(
\begin{array}{cccccccc}
 0.4349 & -0.0167 & -0.0018 & 0. & -0.0005 & 0.0007 & 0.0011 & -0.016 \\
 -0.0167 & 0.4314 & -0.0184 & 0.0003 & 0.0009 & -0.0003 & 0.0024 & -0.0001 \\
 -0.0018 & -0.0184 & 0.4332 & -0.0169 & 0.0016 & 0.0003 & -0.0017 & 0.0007 \\
 0. & 0.0003 & -0.0169 & 0.4305 & -0.0165 & 0.0024 & -0.001 & -0.0028 \\
 -0.0005 & 0.0009 & 0.0016 & -0.0165 & 0.4348 & -0.0159 & 0.0008 & 0.0005 \\
 0.0007 & -0.0003 & 0.0003 & 0.0024 & -0.0159 & 0.4339 & -0.0175 & 0.0013 \\
 0.0011 & 0.0024 & -0.0017 & -0.001 & 0.0008 & -0.0175 & 0.4341 & -0.0193 \\
 -0.016 & -0.0001 & 0.0007 & -0.0028 & 0.0005 & 0.0013 & -0.0193 & 0.4331 \\
\end{array}
\right).
$$

Our theory says that the main diagonal should be $\kappa = \frac{1 - e^{-2}}{2} \sim 0.4323$
and the upper and lower off-diagonals and corners should be $-\frac{1}{e}\sinh (1/22) \sim
-0.01673$. This is  reflected in both computations of the transformed covariance above. Notice that
all other entries are less than $0.01$ in magnitude. Finally we
compute the inverse of the transformed covariance matrix $\Sigma_\pi$. This is given by
$$
\left(
\begin{array}{cccccccc}
 2.317 & 0.0895 & -0.0006 & 0. & 0. & 0. & -0.0006 & 0.0895 \\
 0.0895 & 2.317 & 0.0895 & -0.0006 & 0. & 0. & 0. & -0.0006 \\
 -0.0006 & 0.0895 & 2.317 & 0.0895 & -0.0006 & 0. & 0. & 0. \\
 0. & -0.0006 & 0.0895 & 2.317 & 0.0895 & -0.0006 & 0. & 0. \\
 0. & 0. & -0.0006 & 0.0895 & 2.317 & 0.0895 & -0.0006 & 0. \\
 0. & 0. & 0. & -0.0006 & 0.0895 & 2.317 & 0.0895 & -0.0006 \\
 -0.0006 & 0. & 0. & 0. & -0.0006 & 0.0895 & 2.317 & 0.0895 \\
 0.0895 & -0.0006 & 0. & 0. & 0. & -0.0006 & 0.0895 & 2.317 \\
\end{array}
\right).
$$
As expected from the theory in Section~\ref{sec:inv}, the diagonals entries should be $1/\kappa \sim
2.314$ and the upper and lower off-diagonal entries and corners should be 
$\lambda/\kappa^{2}(1/22) \sim 0.0895$; all other entries are less than $0.01$ as predicted.

The precision and covariance matrices before and after the transformation are shown in gray scale in
Figure~\ref{fig:chain}. Visually, the starting precision and transformed precision matrices are
almost identical.
\begin{figure}[h!t]   \begin{center}
    \begin{subfigure}{.4\textwidth}
        \includegraphics[width=\textwidth]{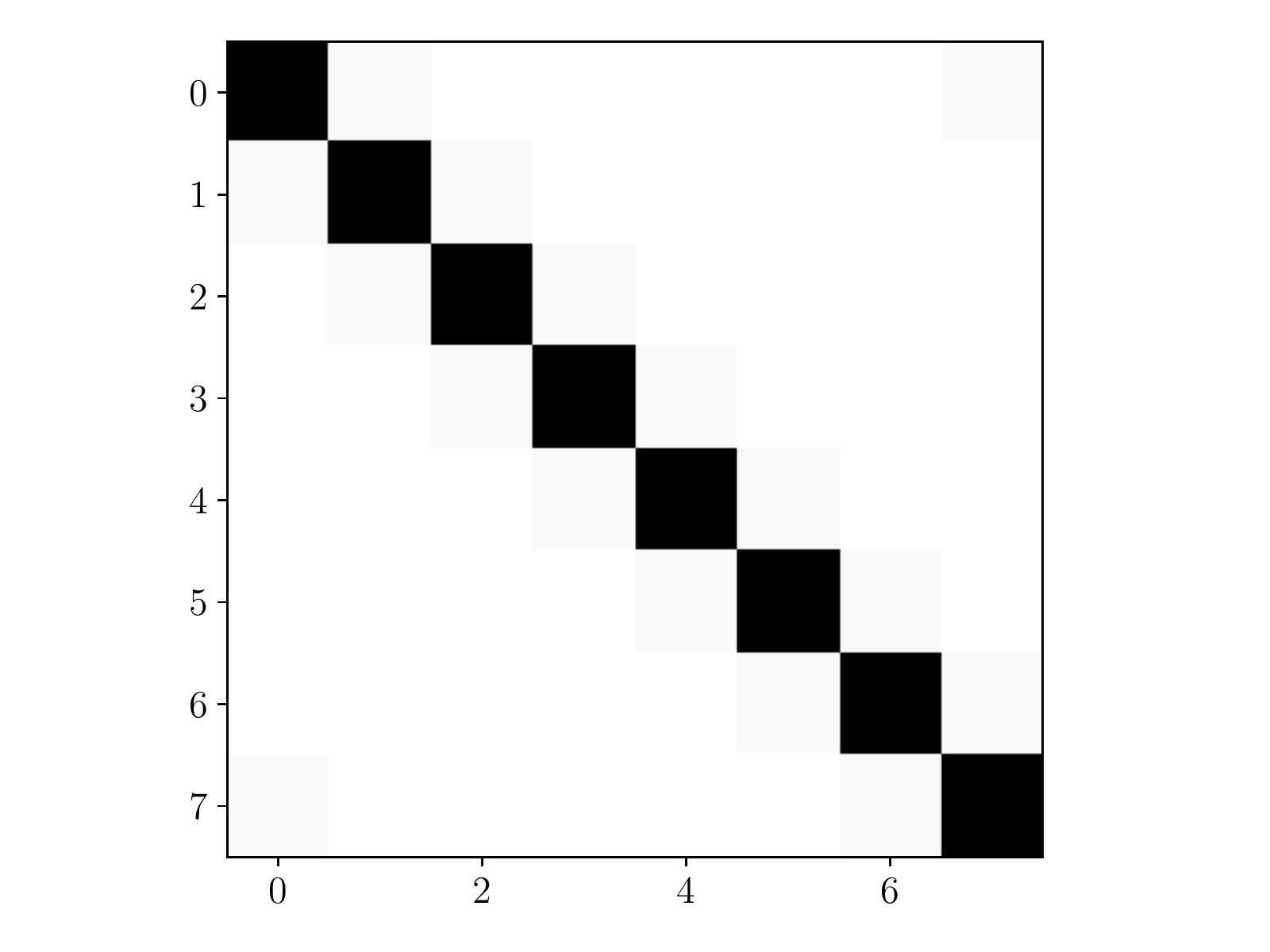} 
        \caption{Starting precision matrix.}
    \end{subfigure}
    \begin{subfigure}{.4\textwidth}
        \includegraphics[width=\textwidth]{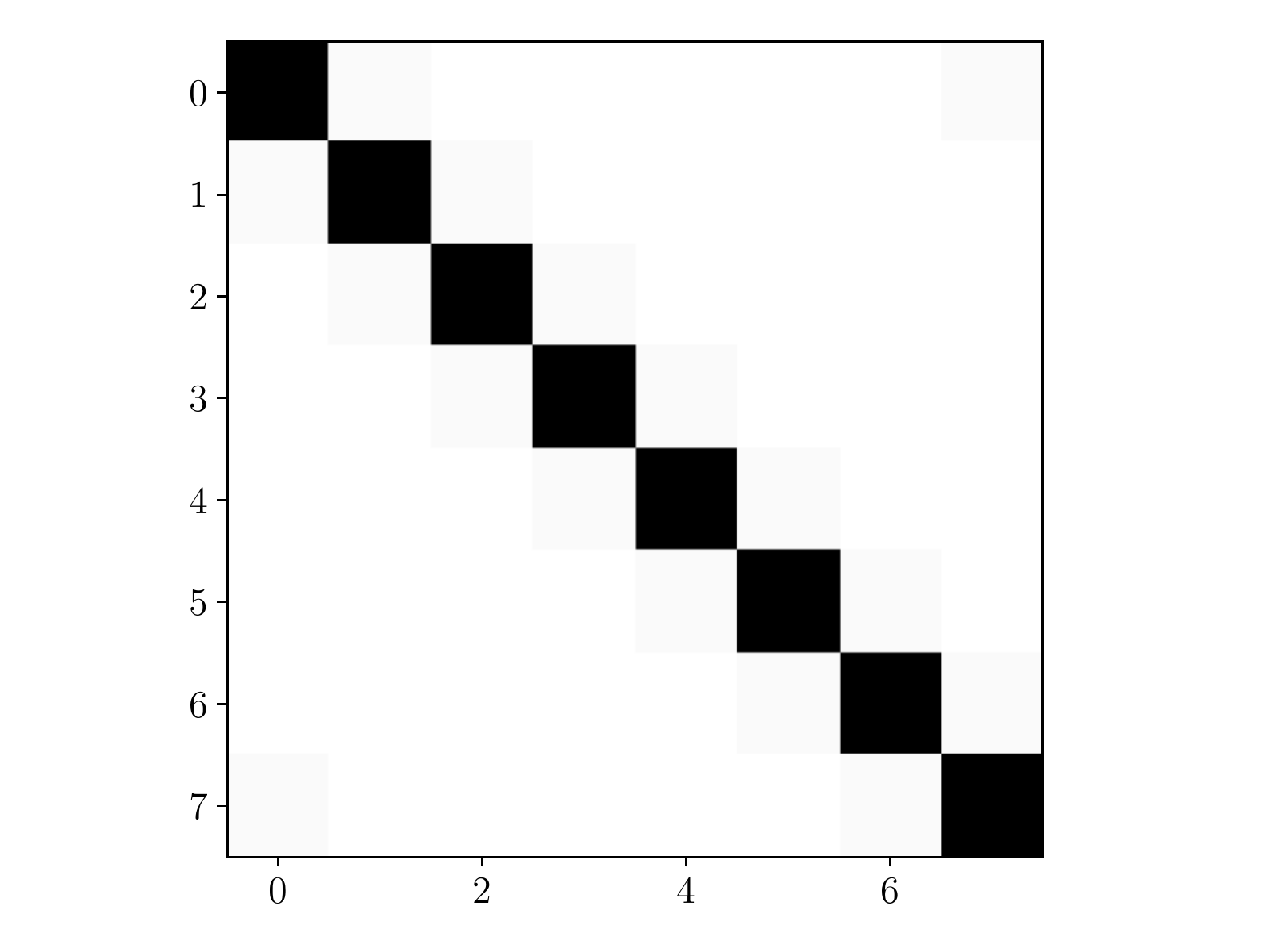} 
        \caption{Starting covariance matrix.}
    \end{subfigure}
    \begin{subfigure}{.4\textwidth}
        \includegraphics[width=\textwidth]{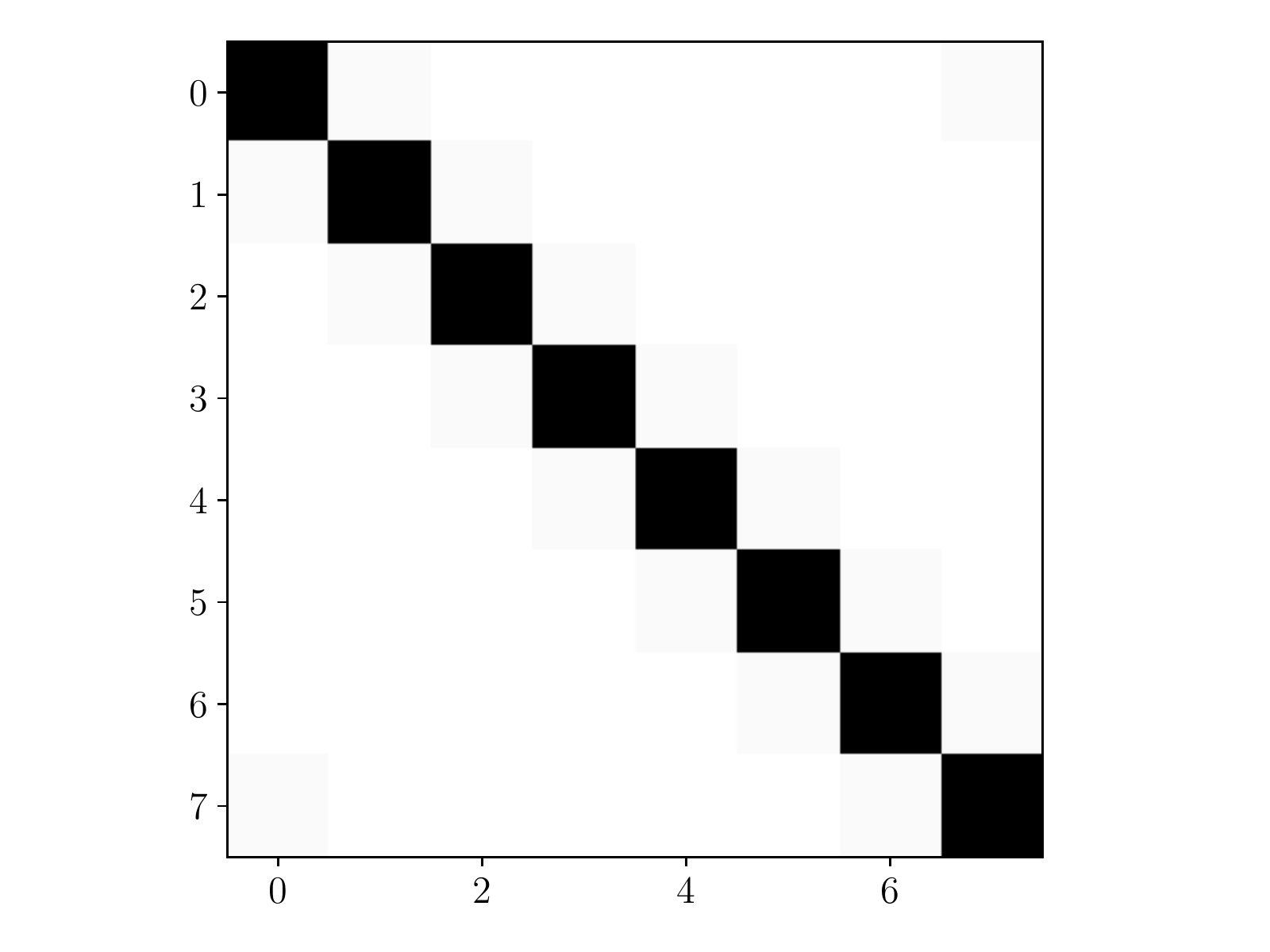} 
        \caption{Transformed precision matrix.}
    \end{subfigure}
    \begin{subfigure}{.4\textwidth}
        \includegraphics[width=\textwidth]{rawfigs/chain-cov2.pdf} 
        \caption{Transformed covariance matrix.}
    \end{subfigure}\caption{Matrices for the chain graph. Gray scale represents absolute values of
    matrix entries. \label{fig:chain}}
\end{center}\end{figure}

\subsection{Star graph}
Let $\Gamma_\rho = \begin{bmatrix} 1 &b^T \\b & I \end{bmatrix}$, where $b$ is an $(d-1)$-vector with $\|b\|^2 <
    1$. Then the inverse of $\Gamma_\rho$ is
    \[\Gamma_\rho^{-1} = \frac{1}{1 - \|b\|^2}\begin{bmatrix} 1 & -b^T \\ -b & (1- \|b\|^2) I +
    bb^T\end{bmatrix} .\]
    
An example for $\Gamma_\rho$ is
$$\left(
\begin{array}{ccccc}
 1 & \frac{1}{11} & \frac{1}{11} & \frac{1}{11} & \frac{1}{11} \\
 \frac{1}{11} & 1 & 0 & 0 & 0 \\
 \frac{1}{11} & 0 & 1 & 0 & 0 \\
 \frac{1}{11} & 0 & 0 & 1 & 0 \\
 \frac{1}{11} & 0 & 0 & 0 & 1 \\
\end{array}
\right),
$$
and its inverse is
$$
\left(
\begin{array}{ccccc}
 1.0342 & -0.094 & -0.094 & -0.094 & -0.094 \\
 -0.094 & 1.0085 & 0.0085 & 0.0085 & 0.0085 \\
 -0.094 & 0.0085 & 1.0085 & 0.0085 & 0.0085 \\
 -0.094 & 0.0085 & 0.0085 & 1.0085 & 0.0085 \\
 -0.094 & 0.0085 & 0.0085 & 0.0085 & 1.0085 \\
\end{array}
\right).
$$
After applying the transformation $f(x) = \sin(x)$, and using theorem~\ref{thm:main}, we have the
transformed covariance $\Sigma_\pi$:
$$
\left(
\begin{array}{ccccc}
 0.4368 & -0.0339 & -0.0339 & -0.0339 & -0.0339 \\
 -0.0339 & 0.4335 & 0.0031 & 0.0031 & 0.0031 \\
 -0.0339 & 0.0031 & 0.4335 & 0.0031 & 0.0031 \\
 -0.0339 & 0.0031 & 0.0031 & 0.4335 & 0.0031 \\
 -0.0339 & 0.0031 & 0.0031 & 0.0031 & 0.4335 \\
\end{array}
\right).
$$
Once again the diagonal entries should be close to $\kappa = \frac{1}{e}\sinh(1) \sim 0.4323$. The
first non-diagonal row and column entries should have magnitude 
$\frac{1}{e}\sinh(0.094) \sim 0.0346,$ which they do. 
Finally, the precision matrix $\Gamma_\pi$ is given by 
$$
\left(
\begin{array}{ccccc}
 2.3451 & 0.1795 & 0.1795 & 0.1795 & 0.1795 \\
 0.1795 & 2.3212 & -0.0025 & -0.0025 & -0.0025 \\
 0.1795 & -0.0025 & 2.3212 & -0.0025 & -0.0025 \\
 0.1795 & -0.0025 & -0.0025 & 2.3212 & -0.0025 \\
 0.1795 & -0.0025 & -0.0025 & -0.0025 & 2.3212 \\
\end{array}
\right),
$$
which as the reader can check is exactly as predicted. All zero
entries in $\Gamma_\rho$ are less than $0.01$
in absolute value. The four matrices are shown in gray scale in Figure~\ref{fig:star}.
\begin{figure}[h!t]   \begin{center}
    \begin{subfigure}{.4\textwidth}
        \includegraphics[width=\textwidth]{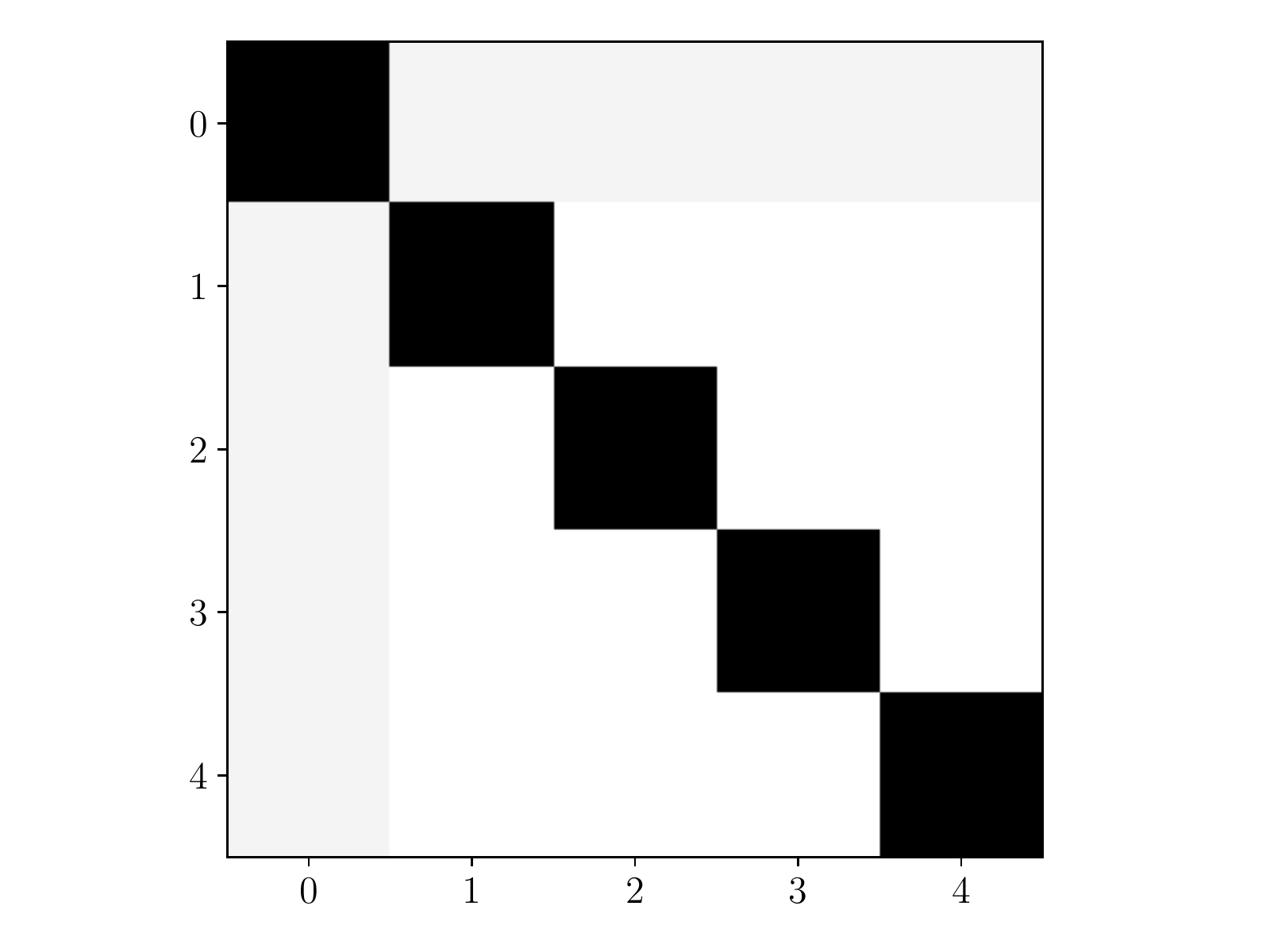} 
        \caption{Starting precision matrix.}
    \end{subfigure}
    \begin{subfigure}{.4\textwidth}
        \includegraphics[width=\textwidth]{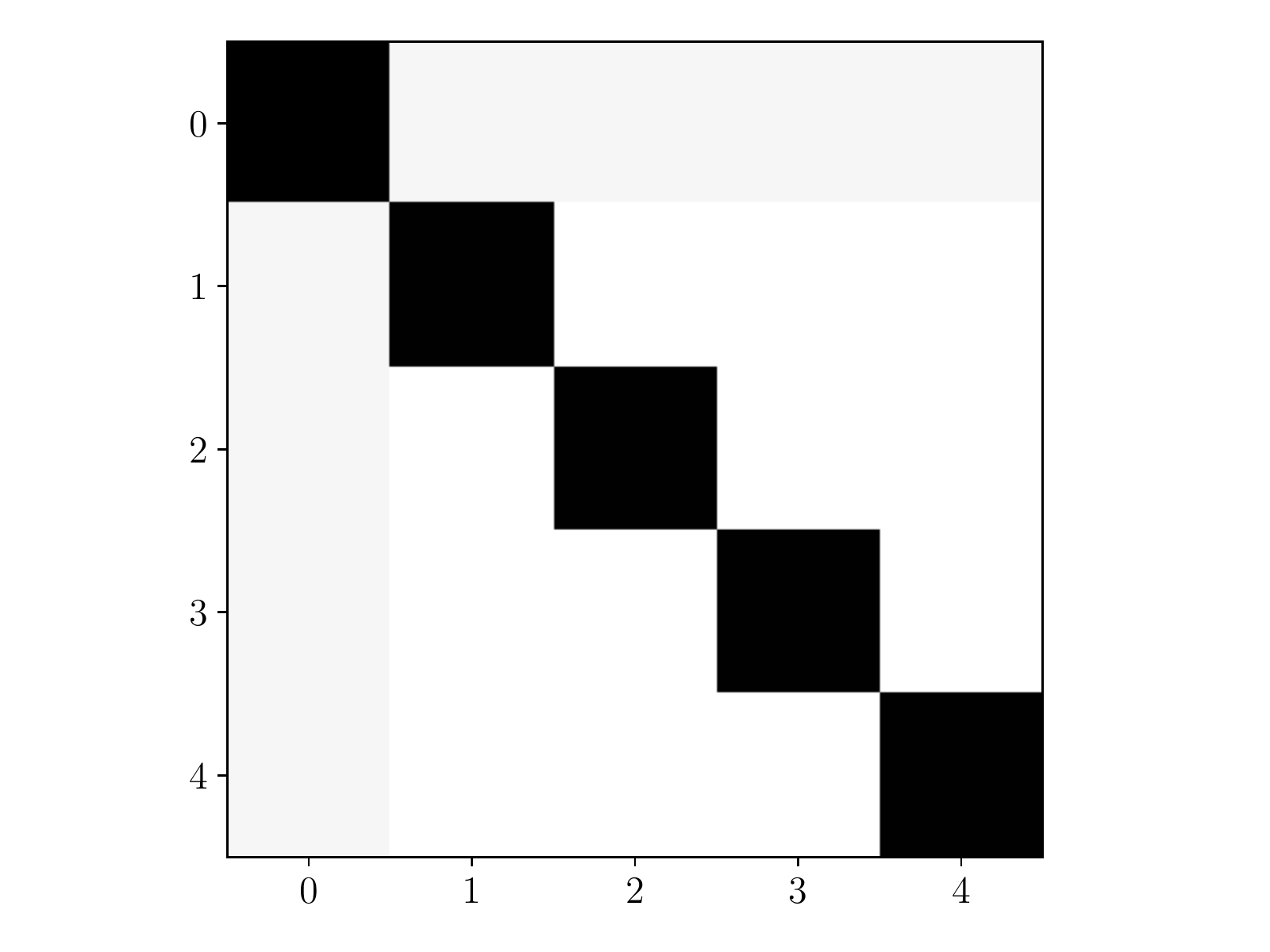} 
        \caption{Starting covariance matrix.}
    \end{subfigure}
    \begin{subfigure}{.4\textwidth}
        \includegraphics[width=\textwidth]{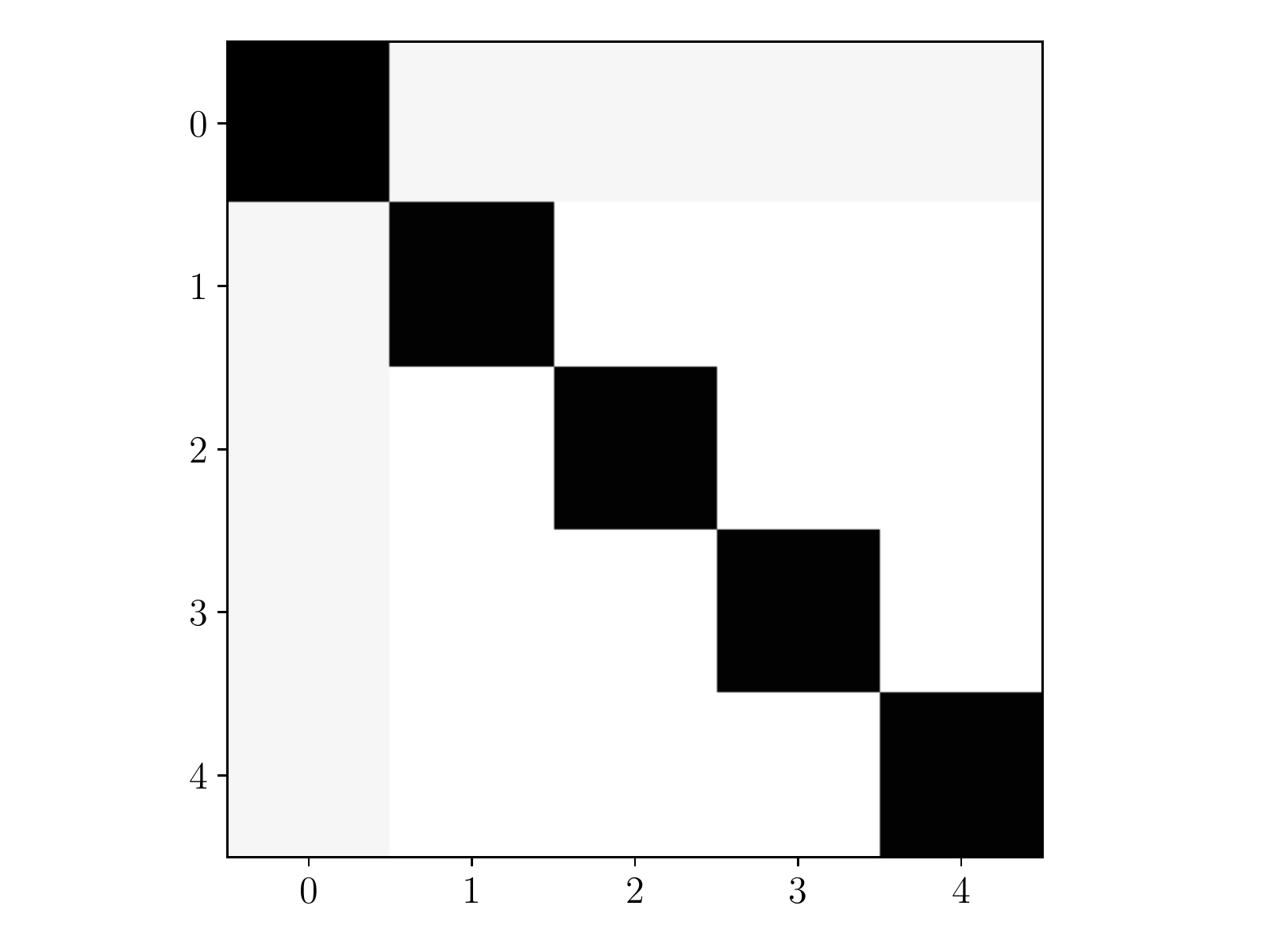} 
        \caption{Transformed precision matrix.}%
    \end{subfigure}
    \begin{subfigure}{.4\textwidth}
        \includegraphics[width=\textwidth]{rawfigs/star-cov2.pdf} 
        \caption{Transformed covariance matrix.}
    \end{subfigure}\caption{Matrices for star graph. Gray scale represents absolute values of
    matrix entries. \label{fig:star}}
\end{center}\end{figure}

\subsection{Grid graph}
For a $3\times3$ grid graph, with nodes ordered across the rows, ones on the diagonal, and $\alpha$
on each edge, the precision matrix has the block Toeplitz form:
$$\alpha \left(
\begin{array}{ccc}
 C & I & 0 \\
 I & C & I \\
 0 & I & C \\
\end{array}
\right)$$
where 
$$C = \left(
\begin{array}{ccc}
 1/\alpha & 1 & 0 \\
 1 & 1/\alpha & 1 \\
 0 & 1 & 1/\alpha \\
\end{array}
\right)$$ and $\alpha < 1/4$ to guarantee positive-definiteness.

Then the covariance is
$$1/\alpha \left[
\begin{array}{ccc}
    C^{-1}(I - (2I-C^2)^{-1}) & (2I-C^2)^{-1} & -C^{-1}(2I-C^2)^{-1} \\
 (2I-C^2)^{-1} & -C(2I-C^2)^{-1} & (2I-C^2)^{-1} \\
    -C^{-1}(2I-C^2)^{-1} & (2I-C^2)^{-1} & C^{-1}(I - (2I-C^2)^{-1}) \\
\end{array}
\right].$$

As a specific example, let $\alpha = 1/6$. Then the original precision $\Gamma_\rho$ is
$$\left(\begin{array}{ccccccccc}
    1& \frac{1}{6} &0 & \frac{1}{6} &0& 0& 0& 0& 0\\
    \frac{1}{6} &1 & \frac{1}{6} &0&\frac{1}{6} &0 & 0 & 0& 0\\
    0& \frac{1}{6}& 1& 0& 0&\frac{1}{6} &0& 0& 0\\
    \frac{1}{6}& 0& 0& 1& \frac{1}{6}& 0& \frac{1}{6} &0& 0\\
    0& \frac{1}{6}& 0& \frac{1}{6}& 1& \frac{1}{6}& 0& \frac{1}{6}& 0\\
    0& 0& \frac{1}{6}& 0& \frac{1}{6}& 1& 0& 0& \frac{1}{6}\\
    0& 0& 0& \frac{1}{6}& 0& 0& 1&\frac{1}{6}& 0\\
    0& 0& 0& 0& \frac{1}{6}& 0& \frac{1}{6}& 1& \frac{1}{6}\\
    0& 0& 0& 0& 0& \frac{1}{6} & 0&\frac{1}{6}& 1\\
\end{array}\right),$$

and the original covariance $\Sigma_\rho$ is
$$\left( \begin{array}{ccccccccc} 
1.0651&-0.1954&0.0357&-0.1954&0.0714&-0.0189&0.0357&-0.0189&0.0063\\
-0.1954&1.1008&-0.1954&0.0714&-0.2143&0.0714&-0.0189&0.0420&-0.0189\\
0.0357&-0.1954&1.0651&-0.0189&0.0714&-0.1954&0.0063&-0.0189&0.0357\\
-0.1954&0.0714&-0.0189&1.1008&-0.2143&0.0420&-0.1954&0.0714&-0.0189\\
0.0714&-0.2143&0.0714&-0.2143&1.1429&-0.2143&0.0714&-0.2143&0.0714\\
-0.0189&0.0714&-0.1954&0.0420&-0.2143&1.1008&-0.0189&0.0714&-0.1954\\
0.0357&-0.0189&0.0063&-0.1954&0.0714&-0.0189&1.0651&-0.1954&0.0357\\
-0.0189&0.0420&-0.0189&0.0714&-0.2143&0.0714&-0.1954&1.1008&-0.1954\\
0.0063&-0.0189&0.0357&-0.0189&0.0714&-0.1954&0.0357&-0.1954&1.0651\\
\end{array}\right).$$

After applying the diagonal transformation $f(x) = \sin(x)$, the covariance matrix $\Sigma_\pi$ is given by
$$\left( \begin{array}{ccccccccc} 
0.4406&-0.0666&0.0123&-0.0666&0.0237&-0.0064&0.0123&-0.0064&0.0022\\
-0.0666&0.4447&-0.0666&0.0238&-0.0703&0.0238&-0.0064&0.0140&-0.0064\\
0.0123&-0.0666&0.4406&-0.0064&0.0237&-0.0666&0.0022&-0.0064&0.0123\\
-0.0666&0.0238&-0.0064&0.4447&-0.0703&0.0140&-0.0666&0.0238&-0.0064\\
0.0237&-0.0703&0.0237&-0.0703&0.4491&-0.0703&0.0237&-0.0703&0.0237\\
-0.0064&0.0238&-0.0666&0.0140&-0.0703&0.4447&-0.0064&0.0238&-0.0666\\
0.0123&-0.0064&0.0022&-0.0666&0.0237&-0.0064&0.4406&-0.0666&0.0123\\
-0.0064&0.0140&-0.0064&0.0238&-0.0703&0.0238&-0.0666&0.4447&-0.0666\\
0.0022&-0.0064&0.0123&-0.0064&0.0237&-0.0666&0.0123&-0.0666&0.4406\\
\end{array}\right).$$
Following our theory in Section~\ref{sec:inv}, the diagonals again scale like $\frac{1}{e}\sinh(1) \sim
0.4323$, and the off-diagonal entries scale like $\frac{1}{e}\sinh\left(\frac{1}{6}\right) \sim 0.0616$.

Finally, the transformed precision $\Gamma_\pi$ is given by
$$\left( \begin{array}{ccccccccc} 
2.3718&0.3325&-0.0099&0.3325&-0.0197&0.0011&-0.0099&0.0011&-0.0001\\
0.3325&2.4035&0.3325&-0.0199&0.3332&-0.0199&0.0011&-0.0108&0.0011\\
-0.0099&0.3325&2.3718&0.0011&-0.0197&0.3325&-0.0001&0.0011&-0.0099\\
0.3325&-0.0199&0.0011&2.4035&0.3332&-0.0108&0.3325&-0.0199&0.0011\\
-0.0197&0.3332&-0.0197&0.3332&2.4392&0.3332&-0.0197&0.3332&-0.0197\\
0.0011&-0.0199&0.3325&-0.0108&0.3332&2.4035&0.0011&-0.0199&0.3325\\
-0.0099&0.0011&-0.0001&0.3325&-0.0197&0.0011&2.3718&0.3325&-0.0099\\
0.0011&-0.0108&0.0011&-0.0199&0.3332&-0.0199&0.3325&2.4035&0.3325\\
-0.0001&0.0011&-0.0099&0.0011&-0.0197&0.3325&-0.0099&0.3325&2.3718\\
\end{array}\right).$$
As expected, the diagonal entries are scaled by $1/\kappa \sim 2.314$ and the off-diagonal entries
are scaled by $\frac{\lambda}{\kappa^2} (\frac{1}{6}) \sim 0.328$.  In this example, all entries that
were zero in the original precision are now less than $0.02$.

Similar to the other examples, Figure~\ref{fig:grid} shows that the starting and transformed
precision matrices of the grid graph are highly similar in nature.
\begin{figure}[h!t]   \begin{center}
    \begin{subfigure}{.4\textwidth}
        \includegraphics[width=\textwidth]{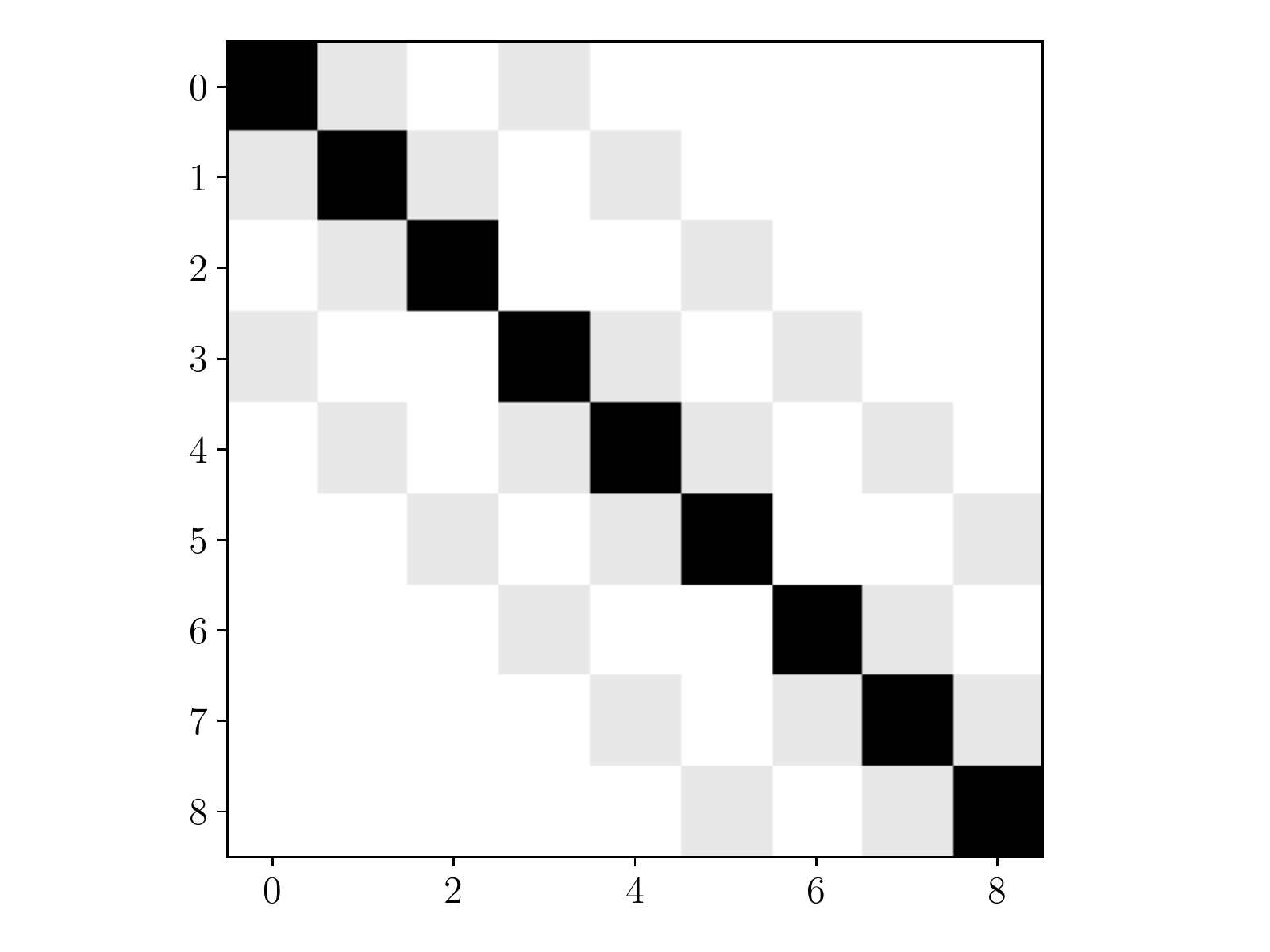} 
        \caption{ Starting precision matrix. \label{fig:prec-p5}}
    \end{subfigure}
    \begin{subfigure}{.4\textwidth}
        \includegraphics[width=\textwidth]{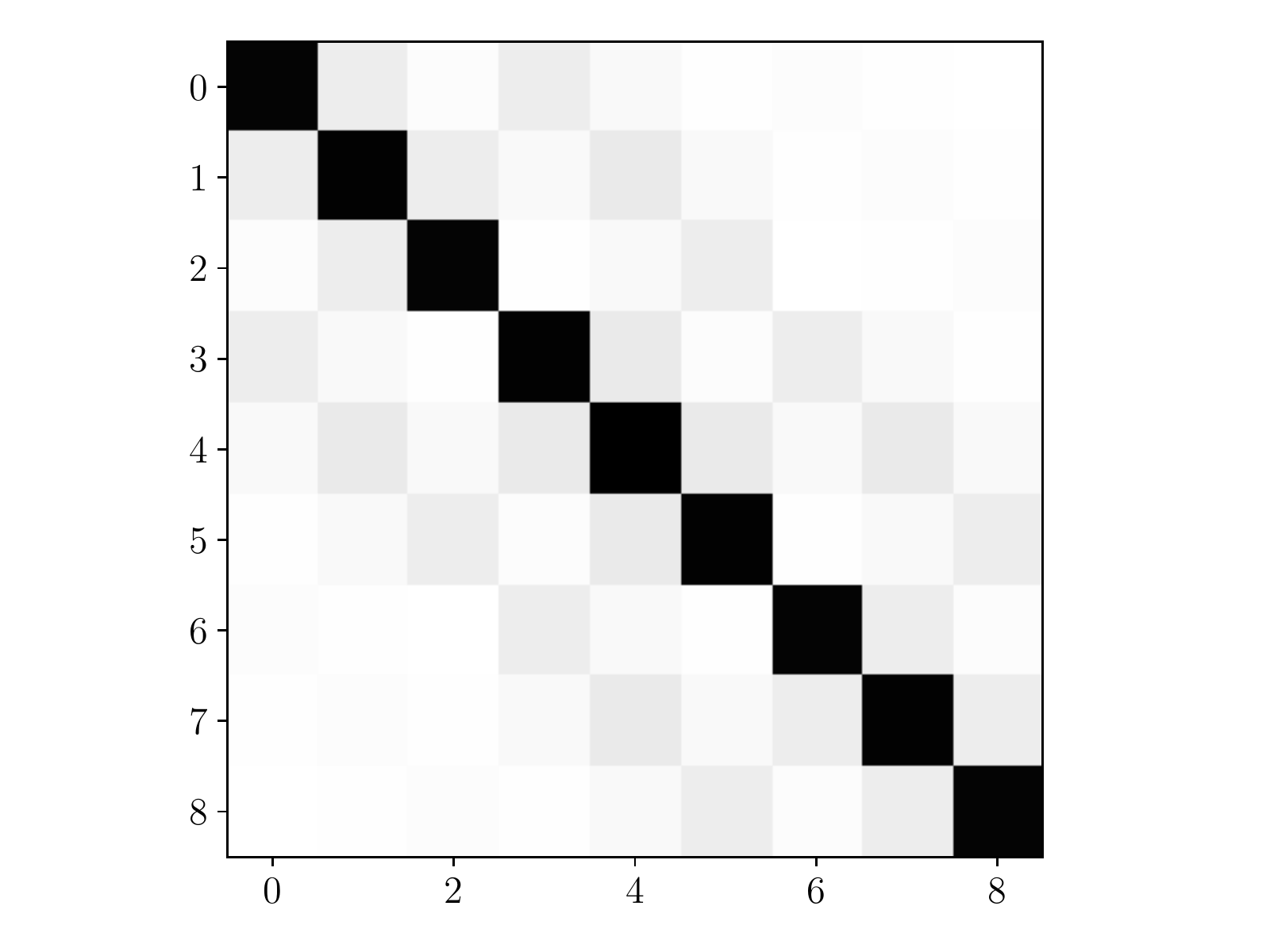} 
        \caption{ Starting covariance matrix.       \label{fig:cov-p5}}
    \end{subfigure}
    \begin{subfigure}{.4\textwidth}
        \includegraphics[width=\textwidth]{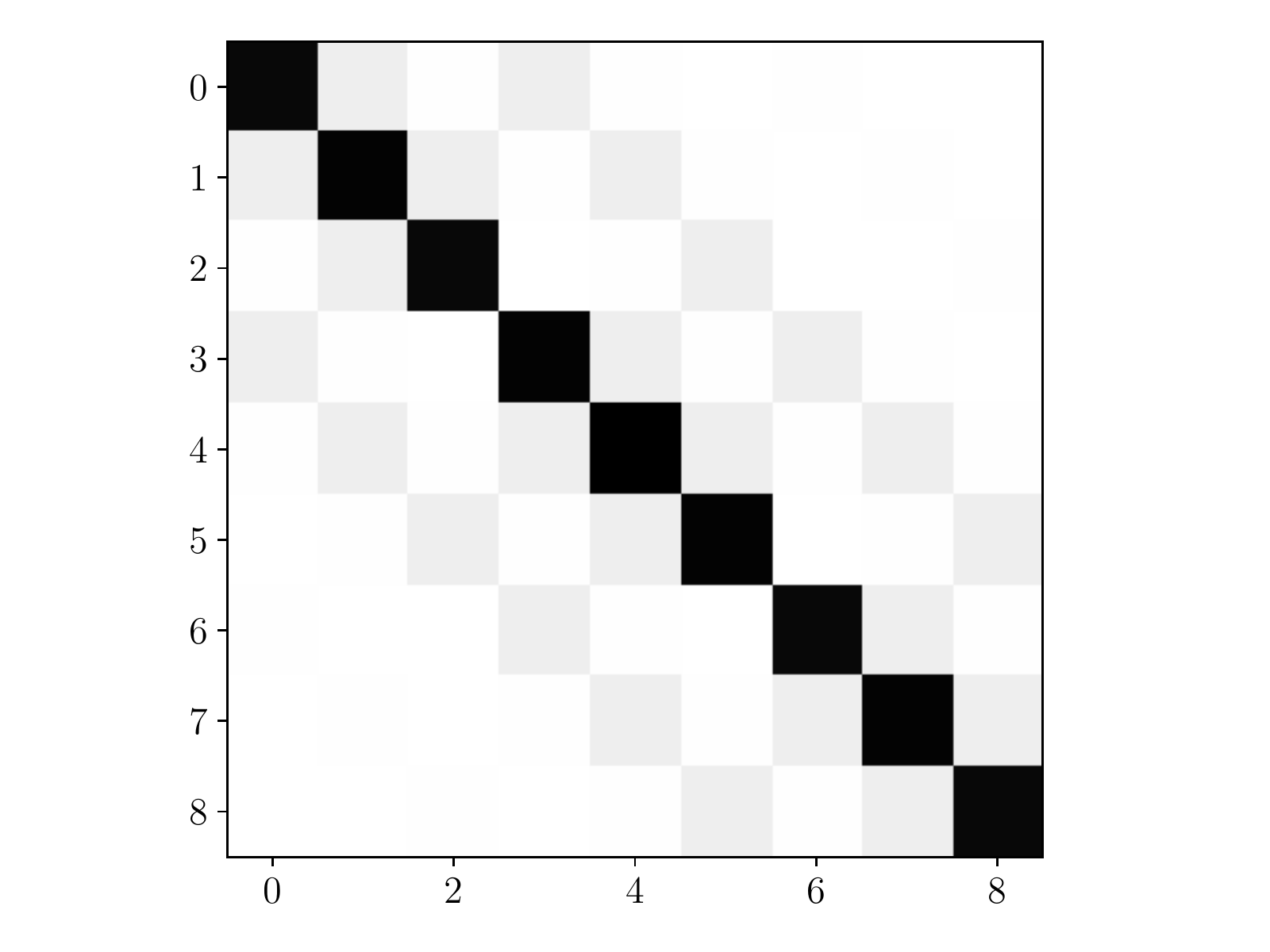} 
        \caption{ Transformed precision matrix.
            \label{fig:prec-p5-transf}}
    \end{subfigure}
    \begin{subfigure}{.4\textwidth}
        \includegraphics[width=\textwidth]{rawfigs/grid-cov2.pdf} 
        \caption{ Transformed covariance matrix.        \label{fig:cov-p5-transf}}
    \end{subfigure}\caption{Matrices for grid graph with $\alpha = 1/6$. Gray scale represent absolute values of
    matrix entries. \label{fig:grid}}
\end{center}\end{figure}


\section{Conclusion}\label{sec:con} This work proves that, under mild assumptions, nonlinear
diagonal transformations of a multivariate normal preserve sparsity, exactly in the transformed
covariance matrix and approximately in the transformed precision matrix. These results substantiate
covariance or inverse covariance estimation as a means to identify independence properties, even
when the distributions are non-Gaussian.

In the field of learning undirected graphical models, nonparanormal distributions are often used as
non-Gaussian test cases for learning algorithms: the graph does not change under the transformation
(and so inherits the same graph that is prescribed for the multivariate normal vector), despite the marginal
distributions being clearly non-Gaussian \cite{liu2009nonparanormal}. Our previous work showed
numerically that assuming---incorrectly---that the nonparanormal data is in fact Gaussian does not
significantly impair graph learning \cite{baptista2021learning}. This was surprising at the time,
but the current work shows how the precision matrix still encodes conditional independence
structure: With respect to the undirected graph, the nonparanormal distribution behaves like a
Gaussian.

Several avenues of future work emerge. First, the analysis here can be extended to even
functions $f$, and to transformations that apply non-identical functions to each element of $X$, i.e., $Y_i =
f^i(X_i)$ where $f^i \neq f^j$.  Second, with this mathematical foundation, a notion of approximate
or weak conditional independence can guide graph learning algorithms, so that sparse graphs are
found up to some acceptable tolerance of weak conditional independence. Third, the results here
could accelerate matrix estimation procures, since a matrix with known sparsity is easier to
estimate than one without.

\begin{appendix}
\section{Extra computations}\label{sec:app}
{\bf Linear term.} 
Here we compute only the coefficient for the linear term in~\eqref{eq:multivariate_moment} with
respect to $\sigma_{ij}$ with $\sigma_{ii} = 1 \,\forall \, i$.
\begin{equation*}
\sum_{\substack{n \geq 2\\\text{even}\,\,n}} \sum_{\substack{p=1\\\text{odd}\,\,p}}^{n-1} \frac{ f^{(p)}(0)f^{(n-p)}(0)p!! (n -p)!!}{p!(n-p)!}.
\end{equation*}
Using the formula for odd integers $p$ $$p!! = \frac{(p+1)!}{(\frac{p+1}{2})! 2^{(p+1)/2}},$$ 
this is 
\begin{equation*}
\sum_{\substack{n \geq 2\\ n \text{ even}}} 2^{-n/2 -1}\sum_{\substack{p=1\\p \text{ odd}}}^{n-1} \frac{(p+1)
    f^{(p)}(0)(n - p+1)f^{(n-p)}(0)}{(\frac{p+1}{2})!(\frac{n-p+1}{2})!}.
\end{equation*}
Let $p = 2k+1$. This becomes
\begin{equation*}
\sum_{\substack{n \geq 2\\ n \text{ even}}}2^{-n/2} \sum_{k =0}^{n/2-1}  \frac{ f^{(2k+1)}(0)(n
    -2k)f^{(n-2k-1)}(0)}{(k)!(\frac{n-2k}{2})!}.
\end{equation*}
Now let $n = 2j$ and we have
\begin{equation} \label{eq:series_expansion_F2}
\sum_{j \geq 1}2^{-(j-1)} \sum_{k=0}^{j-1}  \frac{ f^{(2k+1)}(0)f^{(2j-2k-1)}(0)}{k!(j-k-1)!}.
\end{equation}

Now suppose that 
$$f(x) = \sum_{k= 0}^{\infty} \frac{ f^{(2k+1)}(0)x^{2k+1}}{(2k+1)!}.$$ Define a new function $F$ given by
$$F(x) = \sum_{k= 0}^{\infty} \frac{ f^{(2k+1)}(0)x^{k}}{k!}.$$
The sum
$$\sum_{k=0}^{j-1}  \frac{f^{(2k+1)}(0)f^{(2j-2k+1)}(0)}{k!(j-k)!}$$ is the coefficient of the
$j-1$st term in the series for $F^{2}(x)$. Hence, the series in~\eqref{eq:series_expansion_F2}
equals
\begin{equation} \label{eq:transformed_cov_linearterm}
F^{2}(2^{-1}),
\end{equation}
and this corresponds to the coefficient of the linear term for $\sigma_{ij}$.

\begin{example}
For the function $f(x) = \sin(x)$, we have
$$ f(x) = \sin(x) = \sum_{k = 0}^{\infty} \frac{ (-1)^{k} x^{2k +1}}{(2k+1)!}, $$ and so
$$F(x) = \sum_{k = 0}^{\infty} \frac{ (-1)^{k} x^{k}}{k!} = e^{-x}.$$
Using the expression in~\eqref{eq:transformed_cov_linearterm} we have that the coefficient for the
linear term is $F^{2}(2^{-1}) = 1/e$. This coefficient scales $\sigma_{ij}$ in entry $(i,j)$ of the transformed covariance.
\end{example}

\end{appendix}

\section*{Acknowledgments}
The first author was supported by the Johnson\&Johnson Foundation and
its Women in STEM2D Scholars Program. The second author was supported by the Department of
Energy, Office of Advanced Scientific Computing Research, AEOLUS (Advances in Experimental design,
Optimal control, and Learning for Uncertain complex Systems) center. The third author was supported
in part by the American Institute of Mathematics and the NSF grant DMS-1929334.

\bibliographystyle{plain}
\bibliography{references.bib}

\end{document}